\RequirePackage[l2tabu,orthodox]{nag}
\documentclass[a4paper]{article}

\usepackage{ifxetex}
\ifxetex\else\usepackage[utf8]{inputenc}\fi

\usepackage[british]{babel}

\usepackage{amsmath}
\usepackage{amssymb}
\usepackage{amstext}
\usepackage{amsthm}
\usepackage{mathtools}

\usepackage[T1]{fontenc}
\usepackage[full]{textcomp}

\usepackage{microtype}

\usepackage{csquotes}
\usepackage[backend=bibtex,style=authoryear-icomp,natbib=true,firstinits=true,backref=false,dashed=false,uniquename=init]{biblatex}
\DeclareRobustCommand{\bms}[1]{\boldsymbol{#1}}
\DeclarePairedDelimiterX{\norm}[1]{\lVert}{\rVert}{#1}
\DeclarePairedDelimiterX{\abs}[1]{\lvert}{\rvert}{#1}
\DeclareMathOperator{\sign}{sign}

\newtheorem{lemma}{Lemma}
\newtheorem{theorem}{Theorem}

\usepackage[affil-it]{authblk}

\author{Mårten Wadenbäck}
\affil{Centre for Mathematical Sciences\\Lund University, Sweden}
\title{A Result for Orthogonal Plus Rank-$1$ Matrices}

\begin{document}
    \maketitle
    \begin{abstract}
        In this paper the sum of an orthogonal matrix and an outer product is studied, and a relation between the norms of the vectors forming the outer product and the singular values of the resulting matrix is presented. The main result may be found in Theorem~\ref{thm:main result}.
    \end{abstract}

    \section*{Preliminaries}
    We start by proving Lemma~\ref{lem:inequality involving norms} below, which will later be used in the proof of Lemma~\ref{lem:identity rotation case}.
    The proof of Lemma~\ref{lem:inequality involving norms} relies only on well-known properties of inner products and norms (in particular the \emph{Cauchy-Schwarz inequality}  and the \emph{triangle inequality}).
    The necessary background material may be found in \citet[Chapter~3]{book/trefethen_bau} or \citet[Chapter~6]{book/renardy_rogers}.
    Throughout this paper, $\norm{\bms{x}}$ denotes the Euclidean norm of $\bms{x}$.

    \begin{lemma}
        \label{lem:inequality involving norms}
        Suppose $\bms{x},\bms{y}\in\mathbb{R}^n$ with $\norm{\bms{x}}=1$. Then
        \begin{equation}
            \label{eq:lemma inequality}
            \norm{\bms{x}+\bms{y}}^2+\norm{\bms{y}}\,\norm{2\bms{x}+\bms{y}} \geq 1 \geq \norm{\bms{x}+\bms{y}}^2-\norm{\bms{y}}\,\norm{2\bms{x}+\bms{y}}.
        \end{equation}
    \end{lemma}
    \begin{proof}
        Using the triangle inequality followed by the reverse triangle inequality,
        \begin{equation}
            2 \leq \norm{\bms{y}}+\abs[\big]{2-\norm{\bms{y}}} \leq \norm{\bms{y}}+\norm{2\bms{x}+\bms{y}}.
        \end{equation}
        The first part of \eqref{eq:lemma inequality} now follows, since
        \begin{align}
            1 & \leq 1+\norm{\bms{y}}\big(\norm{\bms{y}}+\norm{2\bms{x}+\bms{y}}-2\big) \\
            & = 1-2\norm{\bms{y}}+\norm{\bms{y}}^2+\norm{\bms{y}}\,\norm{2\bms{x}+\bms{y}} \\
            & \leq 1-2\abs{\bms{x}^T\bms{y}}+\norm{\bms{y}}^2+\norm{\bms{y}}\,\norm{2\bms{x}+\bms{y}} \\
            & \leq 1+2\bms{x}^T\bms{y}+\norm{\bms{y}}^2+\norm{\bms{y}}\,\norm{2\bms{x}+\bms{y}} \\
            & = \norm{\bms{x}+\bms{y}}^2+\norm{\bms{y}}\,\norm{2\bms{x}+\bms{y}}.
        \end{align}
        The second part of \eqref{eq:lemma inequality} follows as
        \begin{align}
            1 & = 1+2\bms{x}^T\bms{y}+\norm{\bms{y}}^2-2\bms{x}^T\bms{y}-\norm{\bms{y}}^2 \\
            & = \norm{\bms{x}+\bms{y}}^2-\bms{y}^T(2\bms{x}+\bms{y}) \\
            & \geq \norm{\bms{x}+\bms{y}}^2-\norm{\bms{y}}\,\norm{2\bms{x}+\bms{y}}.
        \end{align}
        This concludes the proof.
    \end{proof}

    \section*{Main Result}
    The main result in this paper concerns the \emph{singular values} of a special kind of matrices.
    Recommended background material on the \emph{Singular Value Decomposition} (SVD) and its properties may be found in \citet[Section~2.5.3]{book/golub_van_loan} and \citet[Chapters~4--5]{book/trefethen_bau}.
    The present paper proves the following theorem.
    \begin{theorem}
        \label{thm:main result}
        Let $\bms{A} = \bms{Q}+\bms{a}\bms{b}^T$ for an orthogonal $n\times n$ matrix $\bms{Q}$ and $\bms{a},\bms{b}\in\mathbb{R}^n$.
        Let $\sigma_1,\ldots,\sigma_r$ be the singular values of $\bms{A}$, and if $r<n$ additionally let $\sigma_n=0$.
        Then
        \begin{equation}
            \sigma_1-\sign(1+\bms{a}^T\bms{Q}\bms{b})\sigma_n = \norm{\bms{a}}\,\norm{\bms{b}}.
        \end{equation}
    \end{theorem}
    The theorem is obtained from Lemma~\ref{lem:identity rotation case} below by multiplication from any side by $\bms{Q}$ and renaming of the variables.
    This is possible since multiplication by an orthogonal matrix does not change the singular values.
    \begin{lemma}
        \label{lem:identity rotation case}
        Let $\bms{A} = \bms{I}+\bms{a}\bms{b}^T$ with $\bms{a},\bms{b}\in\mathbb{R}^n$.
        Let $\sigma_1,\ldots,\sigma_r$ be the singular values of $\bms{A}$, and if $r<n$ additionally let $\sigma_n=0$.
        Then
        \begin{equation}
            \label{eq:lemma}
            \sigma_1-\sign(1+\bms{a}^T\bms{b})\sigma_n = \norm{\bms{a}}\,\norm{\bms{b}}.
        \end{equation}
    \end{lemma}
    \begin{proof}
        If either $\bms{a}$ or $\bms{b}$ is zero, the proposition follows by inspection, as both sides in \eqref{eq:lemma} evaluate to zero.
        Suppose therefore that neither $\bms{a}$ nor $\bms{b}$ is zero, and introduce
        \begin{equation}
            \bms{x} = \frac{1}{\norm{\bms{a}}}\bms{a}
            \quad\text{and}\quad
            \bms{y} = \norm{\bms{a}}\bms{b}.
        \end{equation}
        Now
        \begin{equation}
            \label{eq:ATA}
            \bms{A}^T\bms{A}=(\bms{I}+\bms{x}\bms{y}^T)^T(\bms{I}+\bms{x}\bms{y}^T)=\bms{I}+\bms{x}\bms{y}^T+\bms{y}\bms{x}^T+\bms{y}\bms{y}^T,
        \end{equation}
        and we may write
        \begin{equation}
            \label{eq:rank revelation}
            \bms{A}^T\bms{A}-\bms{I}=(\bms{x}+\bms{y})(\bms{x}+\bms{y})^T-\bms{x}\bms{x}^T.
        \end{equation}
        We see here that precisely two singular values of $\bms{A}$ are different from $1$, except when $\bms{x}$ and $\bms{y}$ are parallel, in which case only one of them is different from $1$.

        \paragraph{In the parallel case} $\bms{y}=\mu\bms{x}$ it is clear from \eqref{eq:ATA} that any vector orthogonal to $\bms{x}$ is an eigenvector to $\bms{A}^T\bms{A}$ with eigenvalue $1$, and any vector parallel to $\bms{x}$ is an eigenvector with eigenvalue
        \begin{equation}
            \lambda=1+2\mu+\mu^2=(1+\mu)^2=(1+\bms{x}^T\bms{y})^2.
        \end{equation}
        Now
        \begin{equation}
            (\lambda-1-\bms{x}^T\bms{y})^2 = (\lambda-1-\mu)^2
            = (\mu+\mu^2)^2
            = \mu^2(1+\mu)^2
            = \mu^2\lambda,
        \end{equation}
        and
        \begin{align}
            \abs{\mu}^2 = \frac{(\lambda-1-\bms{x}^T\bms{y})^2}{\lambda}
            = \Big(\sqrt{\lambda}-\frac{1+\bms{x}^T\bms{y}}{\sqrt{\lambda}}\Big)^2
            = \big(\sqrt{\lambda}-\sign(1+\bms{x}^T\bms{y})\big)^2.
        \end{align}
        Here $\sqrt{\lambda}$ is equal to either $\sigma_1$ or $\sigma_n$, depending on whether $\sqrt{\lambda}$ is larger or smaller than $1$.
        In both cases it is readily verified that
        \begin{equation}
            \sigma_1-\sign(1+\bms{a}^T\bms{b})\sigma_n
            = \sigma_1-\sign(1+\bms{x}^T\bms{y})\sigma_n
            = \abs{\mu}
            = \norm{\bms{a}}\,\norm{\bms{b}}.
        \end{equation}

        \paragraph{In the non-parallel case} it is clear from \eqref{eq:ATA} that any vector orthogonal to both $\bms{x}$ and $\bms{y}$ is an eigenvector to $\bms{A}^T\bms{A}$ with eigenvalue $1$, and we may choose $n-2$ linearly independent such vectors.
        It follows that the remaining two eigenvectors are linear combinations of $\bms{x}$ and $\bms{y}$.
        One finds that
        \begin{equation}
            \bms{A}^T\bms{A}(\bms{x}+s\bms{y})=(1+\bms{x}^T\bms{y}+s\norm{\bms{y}}^2)\bms{x}+(s+1+s\bms{x}^T\bms{y}+\bms{x}^T\bms{y}+s\norm{\bms{y}}^2)\bms{y},
        \end{equation}
        so in order for $\bms{x}+s\bms{y}$ to be an eigenvector one must have
        \begin{equation}
            \left\{
            \begin{aligned}
                \label{eq:requirements on s}
                1+\bms{x}^T\bms{y}+s\norm{\bms{y}}^2 & =\lambda \\
                s+s\bms{x}^T\bms{y}+1+\bms{x}^T\bms{y}+s\norm{\bms{y}}^2 & = \lambda s
            \end{aligned}
            \right..
        \end{equation}
        This means that
        \begin{equation}
            \begin{aligned}
                s+s\bms{x}^T\bms{y}+1+\bms{x}^T\bms{y}+s\norm{\bms{y}}^2 & = (1+\bms{x}^T\bms{y}+s\norm{\bms{y}}^2)s \Longleftrightarrow \\
                s^2-s-\frac{1+\bms{x}^T\bms{y}}{\norm{\bms{y}}^2} & =0 \Longleftrightarrow \\
                s & = \frac{1}{2}\pm\frac{\norm{2\bms{x}+\bms{y}}}{2\norm{\bms{y}}},
            \end{aligned}
        \end{equation}
        and together with \eqref{eq:requirements on s} this gives the two eigenvalues
        \begin{equation}
            \left\{
            \begin{aligned}
                \lambda_1 & = \frac{1}{2}+\frac{\norm{\bms{x}+\bms{y}}^2}{2}+\frac{\norm{\bms{y}}\,\norm{2\bms{x}+\bms{y}}}{2} \\
                \lambda_2 & = \frac{1}{2}+\frac{\norm{\bms{x}+\bms{y}}^2}{2}-\frac{\norm{\bms{y}}\,\norm{2\bms{x}+\bms{y}}}{2}
            \end{aligned}
            \right..
        \end{equation}
        Using Lemma~\ref{lem:inequality involving norms} and the fact that $\lambda_1 \neq 1$ and $\lambda_2 \neq 1$ we conclude that $\lambda_1=\sigma_1^2$ and $\lambda_2=\sigma_n^2$.
        It can be verified that
        \begin{equation}
            \lambda_1\lambda_2 = (1+\bms{x}^T\bms{y})^2,
        \end{equation}
        and some further computations show that
        \begin{equation}
            (\lambda_1-1-\bms{x}^T\bms{y})^2 = \norm{\bms{y}}^2\lambda_1.
        \end{equation}
        Dividing both sides by $\lambda_1$, we see that
        \begin{equation}
            \begin{aligned}
                \norm{\bms{y}}^2 & = \frac{(\lambda_1-1-\bms{x}^T\bms{y})^2}{\lambda_1}
                = \Big(\sqrt{\lambda_1}-\frac{1+\bms{x}^T\bms{y}}{\sqrt{\lambda_1}}\Big)^2 \\
                & = \Big(\sqrt{\lambda_1}-\sign(1+\bms{x}^T\bms{y})\sqrt{\lambda_2}\Big)^2.
            \end{aligned}
        \end{equation}
        As $\lambda_1>\lambda_2$, it is clear that
        \begin{equation}
            \norm{\bms{y}} = \sigma_1-\sign(1+\bms{x}^T\bms{y})\sigma_n = \sigma_1-\sign(1+\bms{a}^T\bms{b})\sigma_n = \norm{\bms{a}}\,\norm{\bms{b}},
        \end{equation}
        which is what we wanted to prove.
    \end{proof}

    \section*{Acknowledgements}
    The author would like to thank Anders Heyden for feedback throughout the revisions of the manuscript, and Keith Hannabuss for suggesting the formulation~\eqref{eq:rank revelation} in the discussion of a closely related problem.

    \printbibliography
\end{document}